\documentclass[reqno]{amsart}
\usepackage[T1]{fontenc}
\usepackage{amssymb,bbm}
\usepackage{amsmath}
\usepackage{amsfonts}
\usepackage{ifthen}
\usepackage{leftidx}
\usepackage{comment}
\usepackage{mathtools}
\usepackage{marginnote}
\newcommand{\edz}[1]{}
\newtheorem{theorem}{Theorem}[section]
\newtheorem{lemma}[theorem]{Lemma}

\newtheorem{corollary}[theorem]{Corollary}

\newtheorem{definition}[theorem]{Definition}
\newtheorem{example}[theorem]{Example}

\newtheorem{remark}[theorem]{Remark}

\numberwithin{equation}{section}

\newcommand{\const}{\operatorname{const}}

\let\C\undefined
\newcommand{\C}{\mathbb{C}}

{\par\bigskip}
\begin{document}

\title[Godbillon-Vey sequence and Fran\c coise algorithm]{Godbillon-Vey sequence and Fran\c coise algorithm} 
\author[P. Marde\v si\'c]{Pavao Marde\v si\'c}
\address{Universit\'e de Bourgogne, Institute de 
	Math\'ematiques de Bourgogne - UMR 5584 CNRS\\
	Universit\'e de Bourgogne,
	9 avenue Alain Savary,
	BP 47870, 21078 Dijon, \\FRANCE}
\email{mardesic@u-bourgogne.fr}

\author[D. Novikov]{Dmitry Novikov}
\address{Faculty of Mathematics and 
	Computer Science, Weizmann Institute of Science, Rehovot, 7610001\\Israel}
	\email{dmitry.novikov@weizmann.ac.il}
	
\author[L. Ortiz-Bobadilla]{Laura Ortiz-Bobadilla} 
\address{Instituto de Matem\'aticas, Universidad Nacional Aut\'onoma de 
M\'exico 	(UNAM), 	\'Area de la Investigaci\'on Cient\'ifica, Circuito 
exterior, Ciudad 	Universitaria, 04510, Ciudad de M\'exico, M\'exico}
\email{laura@matem.unam.mx}

\author[J. Pontigo-Herrera]{Jessie Pontigo-Herrera}
\address{Instituto de Matem\'aticas, Universidad Nacional Aut\'onoma de 
	M\'exico 	(UNAM), 	\'Area de la Investigaci\'on Cient\'ifica, Circuito 
	exterior, Ciudad 	Universitaria, 04510, Ciudad de M\'exico, M\'exico}
\email{pontigo@matem.unam.mx}

\thanks{This work was supported by Israel Science Foundation grant 1167/17, 
 Papiit (Dgapa UNAM) IN106217, ECOS Nord-Conacyt 249542 and Fordecyt 265667 }

\begin{abstract}
We consider foliations given by deformations $dF+\epsilon\omega$ of exact forms $dF$ in $\mathbb{C}^2$ in a neighborhood of a family of cycles $\gamma(t)\subset F^{-1}(t)$.

In 1996 Fran\c coise gave an algorithm for calculating the first nonzero term of the displacement function $\Delta$ along $\gamma$ of such deformations. This algorithm recalls the well-known Godbillon-Vey sequences discovered in 1971 for investigation integrability of a form $\omega$.
In this paper, we establish the correspondence between the two approaches and translate some results by Casale relating types of integrability for finite Godbillon-Vey sequences to the Fran\c coise algorithm settings. 
	
	\end{abstract}

\maketitle
\section{Introduction}
Let $\gamma_0\subset\mathbb{C}^2$ be a regular curve, $\Sigma$ a trasversal to $\gamma_0$, $F$ a holomorphic function defined on a tubular neighborhood $U\subset \mathbb{C}^2$  of $\gamma_0$, formed by regular curves $\gamma(t)\subset F^{-1}(t)$, $t\in F(\Sigma)$, with $\gamma_0=\gamma(t_0)$.

Consider the integrable foliation $dF=0$ and its holomorphic deformation
  
\begin{equation}\label{eq:deformation}
dF+\epsilon\omega=0
\end{equation}
in $U$. 
We are interested in the displacement function $\Delta$ (holonomy along $\gamma$ minus identity) of \eqref{eq:deformation}.
Here $\Delta(t)$ denotes the holonomy of \eqref{eq:deformation} along $\gamma(t)$.
It can be developped as
\begin{equation}\label{Delta}
\Delta(t)=\sum_{i\ge 1}\epsilon^iM_i(t).
\end{equation}
The functions $M_i(t)$ are called \emph{Melnikov functions}. If $\Delta\equiv0$, this means that \eqref{eq:deformation} has a first integral in a neighborhood of $\gamma(t)$. If not, then  there exists a \emph{first non-zero Melnikov function $M_\mu$.} 

\subsection{Fran\c coise algorithm}
Fran\c coise algorithm allows to compute the first nonzero Melnikov function $M_\mu$. Let us first recall the following classical Lemma.

\begin{lemma}\label{lem:relexactness}
Given a holomorphic one-form $\omega$ and a family of cycles $\gamma(t)\subset\{F^{-1}(t)\}$, the following conditions are equivalent:
\begin{enumerate}
	\item[(i)] The form $\omega$ verifies
\begin{equation}\label{eq:zero integral}
	\int_{\gamma}\omega\equiv0.
\end{equation}
\item[(ii)] There exists a function  $r$ holomorphic in  a neighborhood of $\gamma$ such that
\begin{equation}
dF\wedge (\omega-dr)\equiv 0.
\end{equation}
\item[(iii)] There exist  functions $g$ and $r$ holomorphic in  a neighborhood of $\gamma$ such that
	\begin{equation}\label{eq:*cond}
	\omega=gdF+dr.
	\end{equation}
\end{enumerate}
\end{lemma}
Note that  the functions $g$ and $r$ are univalued in $U$ but in general do not extend to polynomial, nor even univalued functions in $\C^2$.

Recall, the classical result of Poincar\'e and Pontryagin: 	
$$
M_1(t)=-\int_{\gamma(t)}\omega.
$$
If $M_1\equiv 0$, then, by Lemma \ref{lem:relexactness}, 
\begin{equation*}
\omega=g_1dF+dr_1,
\end{equation*}
and in that case, Fran\c coise \cite{F} proves the following theorem (see also \cite{Y}, \cite{JMP1}, \cite{JMP2}, \cite{G}, \cite{GI}, \cite{U1}, \cite{U2}, \cite{P}):
\begin{theorem}\label{francoise}
Let \eqref{Delta} be the displacement function of \eqref{eq:deformation}. Assume that  $M_i(t)\equiv0$, for $i=1,\ldots,k$. Then   $M_{k+1}(t)=(-1)^{k+1}\int_{\gamma(t)}g_{k}\omega$, where $g_0=1$ and
 $g_i$, $r_i$ verify 
\begin{equation}\label{eq:g_iR_i}
g_{i-1}\omega=g_idF+dr_i, \quad i=1,\ldots,k.
\end{equation}
\end{theorem}

The existence of the decomposition \eqref{eq:g_iR_i}, follows by induction from Lemma \ref{lem:relexactness}.

\begin{definition}
\label{francpair}
We call any pair $(g_i,r_i),$  verifying \eqref{eq:g_iR_i} an $i$-th Fran\c coise pair associated to the deformation \eqref{eq:deformation} and call the sequence $(g_i,r_i),$ $i=0,1,\ldots$ a Fran\c coise sequence. We say that the  \emph{length of a Fran\c coise sequence} is $\ell$, if $\ell$ is the smallest index such that $g_{\ell+1}=0.$ If there does not exist such an index, we say that the sequence if of infinite length.
\end{definition}

\medskip

\subsection{Godbillon-Vey sequence}
On the other hand, the classical \emph{Godbillon-Vey sequence} is associated to a foliation defined by a single one form 

\begin{equation}\label{omega}
\omega=0.
\end{equation}
It is a sequence of one-forms $\omega_0=\omega$, $\omega_i$, $i=1,\ldots$ such that the formal one-form 
\begin{equation}\label{Omega}
\Omega=d\epsilon+\omega_0+\sum_{i=1}\frac{\epsilon^i}{i!}\omega_i
\end{equation}
in $\mathbb{C}^2\times\mathbb{C}$ verifies the \emph{formal integrability condition}

\begin{equation}\label{integrability}
\Omega\wedge\tilde{d}\Omega=0.
\end{equation}
Here  $\tilde d=d_\epsilon+d$ denotes the total differential with respect to all variables $x,y,\epsilon$.

Condition \eqref{integrability} is equivalent to
$$
\begin{array}{rcl}
d\omega_0&=&\omega_0\wedge\omega_1,\\
d\omega_1&=&\omega_0\wedge\omega_2,\\
\cdots& &\cdots\\
d\omega_n&=&\omega_0\wedge\omega_{n+1}+
\sum_{k=1}^n {n\choose k}
\omega_k\wedge\omega_{n-k+1}.\\
\end{array}
$$
We say that the \emph{Godbillon-Vey sequence is of length $n$} if the forms $\omega_{k}$ vanish for $k\geq n$.

\begin{definition}\label{DLR} 
Let $K$ be a differential field,  $G$ a function and   $K_G$ the extension of $K$ by $G$. We say that the extension $K_G$ is: \emph{Darboux}, \emph{Liouville}  or \emph{Riccati}, respectively,  if it belongs to a  finite sequence of  field extensions starting from the field $K$. The extensions in each step are either algebraic or given respectively by solutions of the equations  $dG=\eta_0$, $dG=G\eta_1+\eta_0$ or
  $dG=G^2\eta_2+G\eta_1+\eta_0$,  with $\eta_i$ one-forms with coefficients in the corresponding field extensions.
  
  In that case, we call the function $G$ Darboux, Liouville or  Riccati with respect to $K$.
  \end{definition}

In \cite{C}, Casale relates the length $n$ of the Godbillon-Vey sequence to the type of first integral of the foliation given by \eqref{omega}:
\begin{theorem}\hfill
\begin{enumerate}
\item[(i)] There exists a  Godbillon-Vey sequence of length $1$ if and only if \eqref{omega} has a Darboux first integral.
\item[(ii)] There exists a  Godbillon-Vey sequence of length $2$ if and only if \eqref{omega} has a Liouvillian first integral.
\item[(iii)] There exists a  Godbillon-Vey sequence of length $3$ if and only if \eqref{omega} has a Riccati first integral.
\end{enumerate}
\end{theorem}

Here we develop a version of Godbillon-Vey sequences well-adapted to studying a deformation of an integrable foliation given by \eqref{eq:deformation}. 
Recall that on the level $\epsilon=0$ it is integrable (with first integral $F$). 
The Godbillon-Vey sequence gives a condition for verifying if this integrability extends to $\epsilon\ne0$.

We define the form
\begin{equation}\label{eq:Omega}
\Omega=R d\epsilon+(dF+\epsilon\omega)G, 
\end{equation}
with 
\begin{equation}\label{eq:rGF}
G=\sum_{i=0}\epsilon^iG_i,\quad R=\sum_{i=0}\epsilon^i R_{i+1}
\end{equation}
unknown functions and $G_0\equiv 1$.
The form \eqref{eq:Omega} of $\Omega$ comes from the requirement to define the same foliation as \eqref{eq:deformation} on each level $\epsilon=const$. 

We give a relative version of the definition of different types of first integral for the deformation \eqref{eq:deformation}.

\begin{definition}\label{Feps}
We denote by $K_{F,\omega}$ the field associated to the deformation \eqref{eq:deformation}. That is, the smallest differential field in a tubular neighborhood $U$ of a cycle $\gamma_0$ containing the functions given by coefficients of $dF$ and $\omega$.

Let $F_\epsilon=\sum_{i=0}^\ell \epsilon^i F_i$, $\ell<\infty$, be a first integral of \eqref{eq:deformation}. We say that it is Darboux, Liouville or Riccati, respectively, if all $F_i$ are in the corresponding extension of the field  $K_{F,\omega}$.

\end{definition}

\begin{theorem}\label{frobenious}\cite{GV}
There exists a solution $(G,R)$ of the equation 
\begin{equation}\label{eq:frob}
\Omega\wedge \tilde d\Omega=0,
\end{equation} if and only if 
the deformation preserves formal integrability along $\gamma$ i.e. $\Delta\equiv0$. 
\end{theorem}

\begin{proof} Indeed, if there exists a solution $(G,R)$ of \eqref{frobenious}, then by Frobenius theorem, $\Omega$ defines a foliation in a neighborhood of $(\gamma(t),0)$ in  $\C^3$ transversal to $\epsilon=0$. It follows from the existence of this  foliation  that the integrability on the level $\epsilon=0$ is preserved on nearby  levels.  
\end{proof}

We will also consider the Godbillon-Vey equation up to order $k$ with $\Omega, G, R$ given by \eqref{eq:Omega} and \eqref{eq:rGF}:
\begin{equation}\label{Omega_k}
\Omega\wedge\tilde{d}\Omega=0\mod\epsilon^{k+1}.
\end{equation}

\begin{definition}
\label{GVpair}
We call any pair $(G_i,R_i),$  verifying \eqref{Omega_k} an $i$-th \emph{ Godbillon-Vey pair associated to the deformation} \eqref{eq:deformation}, 
 $(G_i,R_i),$ $i=0,1,\ldots$ is the \emph{ Godbillon-Vey sequence associated to the deformation}. We say that the length of a  Godbillon-Vey sequence associated to the deformation is $\ell$,
 if $\ell$ is the smallest index such that $G_{\ell+1}=0$. If there does not exist such an index, we say that the sequence if of infinite length.
\end{definition}

\begin{remark} \label{g_i=0}
Note that the length is associated to any  Fran\c coise sequence or Godbillon-Vey sequence associated to the deformation \eqref{eq:deformation}. 

However, one deformation \eqref{eq:deformation} can have Fran\c coise sequences (or Godbillon-Vey sequence) of different lengths. The minimal length is well defined and one can choose a Fran\c coise sequence so that all $g_k=0$, for $k>\ell$. The same applies for the Godbillon-Vey sequences. 
 
Fran\c coise pairs $(g_i,r_i)$ and Godbillon-Vey pairs $(G_i,R_i)$ exist for all $i=1,2,3,\ldots$ if and only if the deformation preserves integrability along $\gamma$.
\end{remark}

\section{Main theorems}

In this section we state our two main results. The first establishes the relationship between the Fran\c coise pairs and the  Godbillon-Vey pairs associated to the deformation. 
In particular it shows that the minimal length of Fran\c coise sequences and Godbillon-Vey sequences coincide:

\begin{theorem}\label{thm:main} \hfill
\begin{enumerate}
\item[(i)] The Melnikov functions $M_i$, $i=1,\ldots,k$, are identically equal to zero if and only if one can solve 
the  equation 
\begin{equation}\label{Omegak}
\Omega\wedge \tilde{d}\Omega=0\mod \epsilon^{k+1},
\end{equation}

\item[(ii)]
For each choice of the Fran\c coise sequence $(g_i, r_i)$, $i=1,\ldots,k$, the Godbillon-Vey sequence $(G_i, R_i)$, $i=1,\ldots,k$, can be chosen verifying
 the equations 
\begin{equation}\label{eq:giri}
G_i=(-1)^ig_i,\quad R_i=(-1)^{i+1}ir_{i}.
\end{equation}
\item[(iii)] If $\Omega$ verifies \eqref{Omegak} then
\begin{itemize}
	\item[a)]  there exists a function $N=1+\sum_{i=1}^k\epsilon^i n_i$ such that
$$\Omega=N\tilde{d}F_\epsilon\mod \epsilon^{k+1}.$$
Then  the function $F_\epsilon$ is of the form
\begin{equation}\label{F_eps}
F_\epsilon=F+\sum_{i=1}^k(-1)^{i+1}\epsilon^ir_i.
\end{equation}
 and
\begin{equation}\label{dtildeF}
\tilde{d}F_\epsilon
=\tilde{R}d\epsilon + \tilde{G}\left(dF+\epsilon\omega\right).
\end{equation}

\item[b)] Let  $\tilde{G}$  and $\tilde{R}$ be given in
\eqref{dtildeF} and $(G_i, R_i)$, $i=1,\ldots,k$, be its coefficients  as in \eqref{eq:rGF}.
 Then the functions  $(g_i, r_i)$, $i=1,\ldots,k$,  given by \eqref{eq:giri}   are  Fran\c coise pairs.
%
\end{itemize}

\end{enumerate}
\end{theorem}

\edz{In fact, one should say "of weight $\ge k+1$" instead of $\mod \epsilon^{k+1}$", i.e. $d\epsilon$ should have weight $1$}

Our second result gives the type  of local first integral $F_\epsilon$ of the deformation \eqref{eq:deformation} if the  length of its Fran\c coise sequence is finite. The first result is that the first integral is in a finite sequence of extensions of Darboux type. The second shows that it is in a single extension of Liouvillian type.

\begin{theorem}\label{thm:first integral}
Let $\eta_\epsilon=dF+\epsilon\omega$ as in  \eqref{eq:deformation}  be such that there exists a Fran\c coise sequence of finite length $\ell$. 
\begin{enumerate}
\item[(i)]
Then 
\eqref{eq:deformation} admits a univalued first integral which is Darboux with respect to the field $K_{F,\omega}$ of the deformation \eqref{eq:deformation}. 
\item[(ii)]
Then there exists a meromorphic form
$\tilde\eta_{\epsilon}$ verifying the Godbillon-Vey sequence of length $2$:
$$
\begin{aligned}
d\eta_\epsilon&=\eta_\epsilon\wedge\tilde{\theta_\epsilon}\\
d\tilde\theta_\epsilon&=0,\\
\end{aligned}
$$
such that there exists a (possibly multivalued) first integral $\tilde{F}_\epsilon$ of \eqref{eq:deformation} verifying 
$$
d\tilde{F}_\epsilon=f\eta_\epsilon,
$$
where 
$$
df=f\tilde{\theta}_\epsilon
$$
is a (possibly multivalued) function in a tubular neighborhood  $U$ of the cycle $\gamma_0$.

In particular, the function $f$ belongs to a Liouville extension of $K_{F,\omega}$ and $\tilde{F}_\epsilon$ belongs to a  Darboux extension
of this Liouville extension.  
\end{enumerate}

\end{theorem}

\begin{remark}
Note that we are restricting our study to a tubular neighborhood $U$ of a cycle $\gamma_0$. A first integral $F_\epsilon$ which is Darboux in $U$ can be more complicated (Liouville, Riccati,...) when studied globally. 
\end{remark}

\begin{remark} In Theorem \ref{thm:first integral} (ii) we prove in particular that if the deformation \eqref{eq:deformation} has a finite Fran\c coise sequence, then it has a Liouvillian first integral. The converse is an interesting question. 
 \end{remark}

\begin{remark} In Theorem \ref{thm:first integral} we suppose that \eqref{eq:deformation} has a Fran\c coise sequence of finite order. What happens in the case of  $\ell=\infty$? In particular, is it possible to give  a condition assuring that a deformation \eqref{eq:deformation} has a Liouville or a Riccati first integral in these terms?
\end{remark}

\section{Proof of Theorem \ref{thm:main}}
\begin{proof}
We first prove the direct implication of the statement (i), the converse will follow from (iii)(b).
If the functions $M_i$ identically vanish for  $i=1,\ldots,k$, then one can build a first integral $F_\epsilon$ of $dF+\epsilon \omega$ $\mod \epsilon^{k+1}$ in the following way: extend $F$ to transversal $\Sigma$ to $\gamma\times \{0\}$ in $\C^3=\C^2_{x,y}\times \C_\epsilon$ as $F(x,y,\epsilon)=F(x,y)$, and extend it to a neighborhood $U$ of $\gamma\times\{0\}$ in $\C^3$ by the flow. The extension $F_\epsilon$ is a multivalued function, but different branches of $F_\epsilon$ agree $\mod \epsilon^{k+1}$ on $\Sigma$ by assumption, and therefore everywhere in $U$.
In other words, in the decomposition $F_\epsilon=F+\sum_{i\ge 1}\epsilon^i F_i$, $F_i=F_i(x,y)$, the functions $F_i$ are univalued for $i=1,\ldots,k$.

This implies that in the decomposition
$$
\tilde{d}F_\epsilon=\left(F_\epsilon\right)'_\epsilon d\epsilon
+\left(dF+\epsilon\omega\right)G,
$$
the coefficients $\left(F_\epsilon\right)'_\epsilon, G$ are univalued modulo terms of order $\ge k+1$, and we take 
 $\Omega:= j^{k-1}_\epsilon \left(F_\epsilon\right)'_\epsilon d\epsilon+
 \left(dF+\epsilon\omega\right)j^k_\epsilon G$, where $j^k_\epsilon$ denotes the $k$-th jet with respect to $\epsilon$.  Then $\Omega$ verifies \eqref{Omegak}.

Using \eqref{eq:g_iR_i}, the proof of (ii) follows from the computation:
\begin{align*}
(dF+\epsilon\omega)\left(1+\sum_{i=1}^k (-1)^i\epsilon^i g_i\right) =&
dF+\sum_{i=1}^k\epsilon^i\left((-1)^ig_i dF+(-1)^{i-1}g_{i-1}\omega\right)\\
=&dF+\sum_{i=1}^k\epsilon^i(-1)^{i-1}dr_{i}\mod \epsilon^{k+1},
\end{align*}
where $g_0\equiv1$. Therefore, by \eqref{eq:giri} and \eqref{eq:rGF}, 
\begin{align*}
\Omega=Rd\epsilon+G(dF+\epsilon\omega)=&\left(\sum_{i=1}^k(-1)^{i-1}i\epsilon^{i-1}r_i\right)d\epsilon+dF+\sum_{i=1}^k\epsilon^i(-1)^{i-1}dr_{i}\\
=&\tilde{d}\left(F+\sum_{i=1}^k\epsilon^i(-1)^{i-1}r_{i}\right)\mod \epsilon^{k+1}
\end{align*}
is closed up to order $\epsilon^{k+1}$ and therefore satisfies \eqref{Omegak}.

\medskip
We prove statement (iii)(a) and (iii)(b) simultaneously by induction.
We define weights of monomials by posing $w(x)=w(y)=w(dx)=w(dy)=0$ and $w(\epsilon)=w(d\epsilon)=1$, so $d$ and $\tilde{d}$ preserve weights.
We will denote by $o_k$ any collection of terms of weight $> k$. In these notations, \eqref{Omegak} is equivalent to 
\begin{equation}\label{eq:integrability k 2}
\Omega\wedge\tilde{d}\Omega=o_{k+1}.
\end{equation}

Let $N_j=1+\sum_{i=1}^j\epsilon^in_i$. We construct the function $N=N_k$ by induction.

Consider first $k=0$. A simple computation shows that
$$
\Omega\wedge \tilde{d}\Omega=dF\wedge d\epsilon\wedge(\omega-dR_1)+o_1,
$$
so, simplifying by $\wedge d\epsilon$,  \eqref{eq:integrability k 2} for $k=0$ is equivalent to
$$
dF\wedge (\omega-dR_1)=0.
$$
By Lemma~\ref{lem:relexactness}, this equation can be solved if and only if $\int_\gamma\omega\equiv0$, i.e. if and only if the first Fran\c coise condition $M_1\equiv0$ is satisfied. Therefore, the existence of $\Omega$ satisfying \eqref{Omegak} for $k=0$ is equivalent to the first Fran\c coise condition, and we can choose $r_1$ in   \eqref{eq:g_iR_i} to be equal to $R_1$,
$$
\omega=dr_1+g_1dF.
$$
Hence,
\begin{equation}
\Omega=r_1d\epsilon+(dF+\epsilon\omega)(1+\epsilon \beta_1)+o_1
\end{equation}
for some function $\beta_1$. Therefore,
$$
\Omega=\left[r_1d\epsilon+(dF+\epsilon\omega)(1-\epsilon g_1)\right]\left(1+\epsilon (\beta_1+g_1)\right)+o_1=N_1\tilde{d}F_{\epsilon,1}+o_1,
$$
where $N_1=1+\epsilon (\beta_1+g_1)$ and $F_{\epsilon,1}=F+\epsilon r_0$.

Now, let $k>0$ and assume \eqref{eq:integrability k 2}. In particular, it means that 
$\Omega\wedge\tilde{d}\Omega=o_k$. By induction, we have  
\begin{equation*}
\Omega=N_{k-1}\tilde{d}F_{\epsilon,k-1}+o_{k-1}, \quad\text{where}\quad F_{\epsilon,k-1}=F+\epsilon r_0-\ldots+(-1)^{k-2}\epsilon^{k-1} r_{k-1}.
\end{equation*}
Define 
$$
\Theta=N_{k-1}^{-1}\Omega=\tilde{d}F_{\epsilon,k-1}+\theta_k+o_{k},
$$
where $\theta_k$ is homogeneous of weight $k$. We have 
$\Theta\wedge \tilde{d}\Theta=\Omega\wedge \tilde{d}\Omega=o_{k}$.

But $\tilde{d}\Theta=\tilde{d}\theta_k$ has weight $k$. Therefore
$$
\Theta\wedge \tilde{d}\Theta=dF\wedge\tilde{d}\theta_k+o_{k}.
$$
Note that $\Theta$  has form \eqref{eq:Omega}, with $G_i, R_i$ as in \eqref{eq:giri} for $i\le k-1$. Separating terms of weight $k$, we get 
$$
\theta_k=\epsilon^{k-1}R_{k}d\epsilon+\epsilon^kG_k dF+(-1)^{k-1}\epsilon^kg_{k-1}\omega.
$$
Therefore

\begin{equation}\label{eq:induction step}
0=dF\wedge\tilde{d}\theta_k=\epsilon^{k-1}dF\wedge d\epsilon\wedge\left((-1)^{k-1}kg_{k-1}\omega-d R_{k}\right).
\end{equation}
As $\Theta$ is a solution of \eqref{Omegak}, this equation is solvable, which, by Lemma~\ref{lem:relexactness}, means that $\int_\gamma g_{k-1}\omega\equiv 0$, i.e. that the $k$-th Melnikov function vanishes identically. Moreover, \eqref{eq:induction step} implies 
$$
kg_{k-1}\omega=(-1)^{k-1}dR_{k}+ kg_k dF,
$$
i.e. Fran\c coise decomposition \eqref{eq:g_iR_i} of $g_{k-1}\omega$ with $k$-th Francoise pair $(g_k, r_k)$, such that $R_{k}=(-1)^{k-1}kr_k$. 

Therefore
\begin{flalign*}
&\Theta=\left(r_0+\ldots+\epsilon^{k-1}(-1)^{k-1}kr_k\right)d\epsilon\\
&\qquad+\left(dF+\epsilon\omega\right)\left(1+\ldots+(-1)^{k-1}\epsilon^{k-1}g_{k-1}+\epsilon^kG_k\right)+o_{k}=\\
&\qquad=\left(1+\epsilon^k(G_k+(-1)^{k-1}g_k)\right)\tilde{d}F_{\epsilon,k}+o_{k},
\end{flalign*}
where $F_{\epsilon,k}=F+\epsilon r_0-\ldots+(-1)^{k-1}\epsilon^{k} r_{k}$,
and
$$
\Omega=N_k\tilde{d}F_{\epsilon,k}+o_{k},\quad N_k=N_{k-1}\left(1+\epsilon^k(G_k+(-1)^{k-1}g_k)\right),
$$
as required.
\end{proof}

\section{Proof of Theorem \ref{thm:first integral}}

\begin{proof}
Proof of (i): 
Let $(g_i,r_i)$, $i=0,1,2,\ldots$ be a Fran\c coise sequence and assume that $g_i=0$, for $i\geq \ell+1$ (see Remark \ref{g_i=0}).
Let 
\begin{equation}\label{eta,G,F}
\eta_\epsilon=dF+\epsilon\omega, \quad G=\sum_{i=0}^\ell(-1)^{i}\epsilon^ig_i, \quad F_\epsilon=F+\sum_{i=1}^\ell(-1)^{i+1}\epsilon^ir_i.
\end{equation}

It follows from the definition of Fran\c coise pairs \eqref{eq:g_iR_i} that
\begin{equation}\label{GdF}
G\eta_\epsilon=dF_\epsilon.
\end{equation}
Differentiating \eqref{eq:g_iR_i} and dividing by $dF$ (that is, applying the Gelfand-Leray derivative), one obtains
$$
dg_i=\frac{dg_{i-1}\wedge\omega }{dF}+g_{i-1}\frac{d\omega}{dF}=:\eta_{i-1}.
$$
By induction, from the Definition \ref{DLR},  $g_i$ is Darboux, for $i=0,\ldots,\ell$. It now follows from \eqref{eq:g_iR_i} that $r_i$, $i=1,\ldots,\ell$, is Darboux as well and by Definition \ref{Feps}, the first integral $F_\epsilon$ is Darboux with respect to the field $K_{F,\omega}$. 

\medskip
Proof of (ii):
Let $\eta_\epsilon$, $G$ and $F_\epsilon$ be as in \eqref{eta,G,F}.  Now from \eqref{GdF}  it follows that
$$
d\eta_\epsilon=\frac{dG^{-1}}{G^{-1}}\wedge G^{-1}dF_\epsilon=\theta_\epsilon\wedge\eta_\epsilon, \quad \text{for  } \theta_\epsilon=\frac{dG^{-1}}{G^{-1}}.
$$
Hence, $d\theta_\epsilon=0$. The two equations together 
give a Godbillon-Vey sequence of length 2 in a Liouville extension of the space of forms with coeffients in $K_{F,\omega}$ in $(x,y)\in U$ and holomorphic with respect to the parameter $\epsilon$. 

Now Singer's theorem \cite{S} (see also \cite{Cthesis}) gives that there exists a form $\tilde{\theta}_\epsilon$ with coeffcients in $K_{F,\omega}$ (holomorphic with respect to $\epsilon$) verifying the same Godbillon-Vey equations:
$$
\begin{aligned}
d\eta_\epsilon&=\eta_\epsilon\wedge\tilde{\theta_\epsilon},\\
d\tilde{\theta}_\epsilon&=0.
\end{aligned}
$$
That is $\tilde{\theta}_\epsilon$ is closed. Hence, there exists a (possibly multivalued) function $f$ defined in $U$ such that 
$$
df=f\tilde{\theta}_\epsilon.
$$
One verifies that the form $f\tilde{\theta}_\epsilon$ is closed. 
This means that there exists a (possibly multivalued) function
$\tilde{F}_\epsilon$ verifying
$$
d\tilde{F}_\epsilon=f\eta_\epsilon.
$$
\end{proof}

\section{Classical Godbillon-Vey sequences and examples}

Let $\Omega$ be the form given by \eqref{eq:Omega}. 
We apply the classical Godbillon-Vey condition \eqref{integrability} to the form 
$$
\frac{\Omega}{R}=d\epsilon+\eta_0+\epsilon\eta_1+\ldots+\tfrac{\epsilon^i}{i!}\eta_i+\ldots
$$

Comparing to the closed form \eqref{eq:Omega}, we conclude that the forms $\eta_1, \ldots$ defined by
\begin{equation}\label{eq:GV sequence from our form}
\eta_\epsilon=\sum \frac{\epsilon^i}{i!}\eta_i=\frac{dF_\epsilon}{d_\epsilon F_\epsilon}=\frac{dF+\sum_{i=1}^\infty\epsilon^i(-1)^{i-1}dr_i}{\sum_{i=1}^\infty(-1)^{i-1}i\epsilon^{i-1}r_i}
\end{equation}
from a Godbillon-Vey sequence of $\tfrac{dF}{R_1}$:
$$
\eta_0=R_1^{-1}dF,\quad \eta_1=R_1^{-1}(2R_2 dF+dR_1),\quad\ldots,
$$
and the forms 
$$
\tilde{\eta}_1=2R_1^{-1}(R_2 dF+dR_1),\ldots, \tilde{\eta}_i=R_1^{i-1}\eta_i,
$$
form a Godbillon-Vey sequence of $dF$. This sequence could be infinite.

In classical setting one starts from a given foliation $\omega=0$ for $\epsilon=0$, and looks for a simplest perturbation $\omega_\epsilon=0$ such that the form $d\epsilon+\omega_\epsilon$ is integrable. Results of \cite{Cthesis} say that if the foliation $\omega=0$ is  Darboux integrable, Liouville integrable or Riccati integrable, then one can find perturbations such that $\omega_\epsilon$ either does not depend on $\epsilon$ or is polynomial in $\epsilon$ of degree $1$ or $2$, respectively, i.e. that the Godbillon-Vey sequence has finite length. 

In this paper, given a  perturbation \eqref{eq:deformation}, and we construct a one-form $\Omega$ such that its restriction to the planes $\{\epsilon=\const\}$ defines the same foliation as the initial one. In other words, unlike the classical settings, here the perturbation of the  foliation  is almost uniquely prescribed, the only freedom being the coefficients ${G},{R}$ in \eqref{eq:Omega}. Thus the length of the corresponding Godbillon-Vey sequence can be  infinite even if $\omega_\epsilon$ is Liouville integrable for all $\epsilon$.

\begin{example}
	Let $F=x^2+y^2$ and $\omega=y^2dx$. For symmetry reasons, the perturbation \eqref{eq:deformation} is integrable. Computation shows that 
	\begin{equation}
	g_n=\frac{(-1)^n}{n!}x^n,\quad r_n=\frac{(-1)^n}{n!}\left(\frac{2}{n+2}x^{n+2}-x^ny^2\right).
	\end{equation}
	Then a first integral is given by 
	\begin{align*}
	F_\epsilon=& x^2+y^2+\sum_{n=1}^\infty (-1)^{n-1}\epsilon^nr_n
	={{ e}^{\epsilon x}} \left( {y}^{2}+2\,{\frac {x}{\epsilon}}-2\,{\epsilon}^{-2} \right),\\
	\tilde{d}F_\epsilon=&e^{\epsilon x}(\epsilon y^2+2x)dx +2e^{\epsilon x}ydy,
	\end{align*}
	and therefore the Godbillon-Vey forms $\omega_i$ are Taylor coefficients in $\epsilon$ of \begin{equation}
	\left(\frac{\partial}{\partial \epsilon}F_\epsilon\right)^{-1}dF_\epsilon=dF+\sum_{n=1}^\infty\epsilon^i\omega_i.
	\end{equation}
	One can see that this series is not polynomial in $\epsilon$, though the first integral $F_\epsilon$ is of Liouville type. Some authors call this type of functions generalized Darboux.
\end{example}
\begin{example}
	For a trivially integrable perturbation $\omega=gdF$, the Fran\c coise pairs are given by $g_i=g^i$, $i=1,\ldots$, and $r_1\equiv 1$, $r_i=0$ for $i=2,\ldots$. Therefore the first integral is 
	$$
	F_\epsilon=F+\epsilon, \qquad R\equiv 1,\quad  G=1+\sum_{i=1}^\infty (-1)^i\epsilon^ig^i=(1-\epsilon g)^{-1},
	$$
\edz{iz $F_\epsilon$ ok? }	
	and 
	$$
	\Omega=d\epsilon + (dF+\epsilon g dF)(1-\epsilon g)^{-1}=dF_\epsilon.
	$$
\end{example}

\begin{example}
	For a Darboux integrable perturbation \eqref{eq:deformation} with $\omega=F\frac{dr}{r}$ we have 
	$$
	g_i=(-1)^i\frac{(\log r)^i}{i!},\qquad r_i=-Fg_i,
	$$
	so
	$$
	F_\epsilon=F+\sum_{i=1}^\infty (-1)^{i-1}\epsilon^ir_i=
	F+F\sum_{i=1}^\infty\epsilon^i\frac{(\log r)^i}{i!}=Fe^{\epsilon\log r}=Fr^\epsilon.
	$$
	Then 
	$$
	\tilde{d}F_\epsilon=Fr^\epsilon\log r d\epsilon+\left(r^\epsilon dF+\epsilon F r^{\epsilon-1}dr\right)=Fr^\epsilon\log r \left[d\epsilon+\left(\frac{dF}{F\log r}+\epsilon\frac{dr}{r\log r}\right)\right].
	$$
	Therefore the forms 
	$$
	\omega_1=\frac{dr}{r\log r},\quad  \omega_i=0, \qquad i=2,\ldots
	$$ form a 
	Godbillon-Vey sequence for $\frac{dF}{F\log r}$, and hence
	$$
	\tilde{\omega}_1=\frac{dr}{r\log r}+\frac{dF}{F}+\frac{dr}{r},\quad  \tilde{\omega}_i=0, \qquad i=2,\ldots
	$$
	form a 
	Godbillon-Vey sequence for $dF$, so this Godbillon-Vey sequence for $dF+\epsilon\omega$ has length $1$.
\end{example}

\subsection*{Acknowledgements}
 {We would like to thank Sergei Voronin for fruitful discussions.}

\end{document}